\documentclass[twoside,leqno,10pt, A4]{amsart}
\usepackage{amsfonts}
\usepackage{amsmath}
\usepackage{amscd}
\usepackage{amssymb}
\usepackage{amsthm}
\usepackage{amsrefs}
\usepackage{latexsym}
\usepackage{mathrsfs}
\usepackage{bbm}
\usepackage{amscd}
\usepackage{amssymb}
\usepackage{amsthm}
\usepackage{amsrefs}
\usepackage{latexsym}
\usepackage{mathrsfs}
\usepackage{bbm}
\usepackage{enumerate}
\usepackage{graphicx}
\usepackage{color}
\begin{document}

\newtheorem{theorem}[subsection]{Theorem}
\newtheorem{proposition}[subsection]{Proposition}
\newtheorem{lemma}[subsection]{Lemma}
\newtheorem{corollary}[subsection]{Corollary}
\newtheorem{conjecture}[subsection]{Conjecture}
\newtheorem{prop}[subsection]{Proposition}
\newtheorem{defin}[subsection]{Definition}

\numberwithin{equation}{section}
\newcommand{\mr}{\ensuremath{\mathbb R}}
\newcommand{\mc}{\ensuremath{\mathbb C}}
\newcommand{\dif}{\mathrm{d}}
\newcommand{\intz}{\mathbb{Z}}
\newcommand{\ratq}{\mathbb{Q}}
\newcommand{\natn}{\mathbb{N}}
\newcommand{\comc}{\mathbb{C}}
\newcommand{\rear}{\mathbb{R}}
\newcommand{\prip}{\mathbb{P}}
\newcommand{\uph}{\mathbb{H}}
\newcommand{\fief}{\mathbb{F}}
\newcommand{\majorarc}{\mathfrak{M}}
\newcommand{\minorarc}{\mathfrak{m}}
\newcommand{\sings}{\mathfrak{S}}
\newcommand{\fA}{\ensuremath{\mathfrak A}}
\newcommand{\mn}{\ensuremath{\mathbb N}}
\newcommand{\mq}{\ensuremath{\mathbb Q}}
\newcommand{\half}{\tfrac{1}{2}}
\newcommand{\f}{f\times \chi}
\newcommand{\summ}{\mathop{{\sum}^{\star}}}
\newcommand{\chiq}{\chi \bmod q}
\newcommand{\chidb}{\chi \bmod db}
\newcommand{\chid}{\chi \bmod d}
\newcommand{\sym}{\text{sym}^2}
\newcommand{\hhalf}{\tfrac{1}{2}}
\newcommand{\sumstar}{\sideset{}{^*}\sum}
\newcommand{\sumprime}{\sideset{}{'}\sum}
\newcommand{\sumprimeprime}{\sideset{}{''}\sum}
\newcommand{\sumflat}{\sideset{}{^\flat}\sum}
\newcommand{\shortmod}{\ensuremath{\negthickspace \negthickspace \negthickspace \pmod}}
\newcommand{\V}{V\left(\frac{nm}{q^2}\right)}
\newcommand{\sumi}{\mathop{{\sum}^{\dagger}}}
\newcommand{\mz}{\ensuremath{\mathbb Z}}
\newcommand{\leg}[2]{\left(\frac{#1}{#2}\right)}
\newcommand{\muK}{\mu_{\omega}}
\newcommand{\thalf}{\tfrac12}
\newcommand{\lp}{\left(}
\newcommand{\rp}{\right)}
\newcommand{\Lam}{\Lambda_{[i]}}
\newcommand{\lam}{\lambda}
\newcommand{\af}{\mathfrak{a}}
\newcommand{\sw}{S_{[i]}(X,Y;\Phi,\Psi)}
\newcommand{\lz}{\left(}
\newcommand{\pz}{\right)}
\newcommand{\bfrac}[2]{\lz\frac{#1}{#2}\pz}
\newcommand{\odd}{\mathrm{\ primary}}
\newcommand{\even}{\text{ even}}
\newcommand{\res}{\mathrm{Res}}
\newcommand{\sumn}{\sumstar_{(c,1+i)=1}  w\left( \frac {N(c)}X \right)}
\newcommand{\lab}{\left|}
\newcommand{\rab}{\right|}
\newcommand{\Go}{\Gamma_{o}}
\newcommand{\Ge}{\Gamma_{e}}
\newcommand{\M}{\widehat}

\theoremstyle{plain}
\newtheorem{conj}{Conjecture}
\newtheorem{remark}[subsection]{Remark}

\makeatletter
\def\widebreve{\mathpalette\wide@breve}
\def\wide@breve#1#2{\sbox\z@{$#1#2$}%
     \mathop{\vbox{\m@th\ialign{##\crcr
\kern0.08em\brevefill#1{0.8\wd\z@}\crcr\noalign{\nointerlineskip}%
                    $\hss#1#2\hss$\crcr}}}\limits}
\def\brevefill#1#2{$\m@th\sbox\tw@{$#1($}%
  \hss\resizebox{#2}{\wd\tw@}{\rotatebox[origin=c]{90}{\upshape(}}\hss$}
\makeatletter

\title[Ratios conjecture for quadratic Hecke $L$-functions in the Gaussian field]{Ratios conjecture for quadratic Hecke $L$-functions in the Gaussian field}

%%\date{\today}
\author[P. Gao]{Peng Gao}
\address{School of Mathematical Sciences, Beihang University, Beijing 100191, China}
\email{penggao@buaa.edu.cn}

\author[L. Zhao]{Liangyi Zhao}
\address{School of Mathematics and Statistics, University of New South Wales, Sydney NSW 2052, Australia}
\email{l.zhao@unsw.edu.au}

\begin{abstract}
 We develope the $L$-functions ratios conjecture with one shift in the numerator
and denominator in certain ranges for the family of quadratic Hecke $L$-functions in the Gaussian field using multiple
Dirichlet series under the generalized Riemann
hypothesis. We also obtain an asymptotical formula for the first moment of central values of the same family of $L$-functions, obtaining an error term of size $O(X^{1/2+\varepsilon})$.
\end{abstract}

\maketitle



\section{Introduction}\label{sec 1}

  The $L$-functions ratios conjecture formulated in \cite[Section 5]{CFZ} by J. B. Conrey, D. W. Farmer and M. R. Zirnbauer gives a general recipe to predict both the main and lower-order terms in the asymptotic formulas for the sum of ratios of products of shifted
$L$-functions. This conjecture has been applied to study a wide variety of important problems such as the density conjecture of N. Katz
and P. Sarnak \cites{KS1, K&S} on the distribution of zeros near the central point of a family of $L$-functions, the mollified moments of $L$-functions, the discrete moments of the Riemann zeta function and its derivatives. A detailed description of these applications can be found in \cite{CS}. \newline

  Some results can be found in the literature concerning the ratios conjecture. For certain ranges of parameters, the ratios conjecture for quadratic $L$-functions was established by H. M. Bui, A. Florea and J. P. Keating \cite{BFK21} over function fields and by  M. \v Cech \cite{Cech1} over $\mq$ by further assuming the generalized Riemann hypothesis (GRH). \newline

  The results of \v Cech \cite{Cech1} were obtained by utilizing the method of double Dirichlet series, a powerful tool that has been
previous deployed to investigate related issues such as moments of central values of families of $L$-functions.  In this paper, we further apply this approach to study the ratios conjecture for quadratic Hecke $L$-functions over the Gaussian field. To state our result, we write $K=\mq(i)$ for the Gaussian field and $\mathcal{O}_K=\mz[i]$ for its ring of integers throughout the paper.  Further $N(n)$ denotes the norm of any $n \in K$ and $\zeta_K(s)$ the Dedekind zeta function of $K$. It is shown in Section \ref{sec2.4} below that every ideal in $\mathcal{O}_K$ co-prime to $2$ has a unique generator congruent to $1$ modulo $(1+i)^3$ which is called primary. \newline

 Let $\chi_m$ be the quadratic symbol $\left(\frac{m}{\cdot} \right)$ defined in Section \ref{sec2.4}, which can be viewed as an analogue in $K$ to the Kronecker symbol. As $\chi_m$ equals $1$ on the group of units of $K$,  we may regard it as a quadratic Hecke character of trivial infinite type and denote the associated $L$-function by $L(s, \chi_m)$.  Furthermore, we use the notation $L^{(c)}(s, \chi_m)$ for the Euler product defining $L(s, \chi_m)$ but omitting those primes dividing $c$. \newline
	
We first establish a result concerning the ratios conjecture with one shift in the numerator
and denominator for the family of Hecke $L$-functions averaged over all quadratic Hecke characters.
\begin{theorem}
\label{Theorem for all characters}
		Assume the truth of GRH. Let $w(t)$ be a non-negative Schwartz function and let $\hat w(s)$ be its Mellin transform.  For any $\varepsilon>0$, $1/2>\Re(\alpha)>0$ and $\Re(\beta)>\varepsilon$, we have
\begin{align}
\label{Asymptotic for ratios of all characters}
\begin{split}	
&\sum_{n\odd}  \frac{L(\tfrac{1}{2}+\alpha,\chi_{(1+i)^2n})}{L(\tfrac{1}{2}+\beta,\chi_{(1+i)^2n})} w\left( \frac {N(c)}X \right) \\
 &=  X\M w(1)\frac{\pi \zeta_K^{(2)}(1+2\alpha)}{8\zeta_K^{(2)}(1+\alpha+\beta)}
\prod_{(\varpi,2)=1}\lz1+\frac{N(\varpi)^{\alpha-\beta}-1}{N(\varpi)^{1+\alpha-\beta}(N(\varpi)^{1+\alpha+\beta}-1)}\pz\\
& \hspace*{0.5cm}  +X^{1-\alpha}\M w(1-\alpha)\frac{\pi^{2\alpha+1}\Gamma(1-2\alpha)\Gamma (\alpha)}{\Gamma (1-\alpha)\Gamma (2\alpha)}\cdot\frac{P( \tfrac{3}{2}-\alpha+\beta)\zeta_K(1-2\alpha)}{\zeta_K(2)\zeta_K(1-\alpha+\beta)}\cdot\frac{2^{\alpha+\beta-2}}{3\cdot 2^{1-\alpha+\beta}-2} +O\lz(1+|\alpha|)^\varepsilon|\beta|^\varepsilon X^{N(\alpha,\beta)+\varepsilon}\pz,
\end{split}
\end{align}
 where
\begin{equation}
\label{Pz}
		P(z)=\prod_{\varpi}\lz1+\frac1{\lz (N(\varpi)^{z-1/2}-1\pz\lz N(\varpi)+1\pz}\pz,
\end{equation} 	
   and
\begin{equation}
\label{Nab}
			N(\alpha,\beta)=\max\left\{1-2\Re(\alpha),1-2\Re(\beta) \right\}.
\end{equation}
\end{theorem}
	
   The above result is consistent with the prediction from the ratios conjecture on the left-hand side of \eqref{Asymptotic for ratios of all characters}, which can be derived following the treatments given in \cite[Section 5]{G&Zhao2022-2}. The only difference is that the ratios conjecture asserts that \eqref{Asymptotic for ratios of all characters} holds uniformly for $|\Re(\alpha)|< 1/4$, $(\log X)^{-1} \ll \Re(\beta) < 1/4$ and $\Im(\alpha), \Im(\beta) \ll X^{1-\varepsilon}$ with an error term $O(X^{1/2+\varepsilon})$. Our result here has the advantage that there is no constraint on imaginary parts of $\alpha$ and $\beta$. \newline

  We shall establish Theorem \ref{Theorem for all characters} following the line of approach in the proof of \cite[Theorem 1.2]{Cech1}. A key ingredient involved is a functional equation of $L$-functions attached to the general quadratic Hecke character given in Proposition \ref{Functional equation with Gauss sums}. \newline
	
  As an application of Theorem \ref{Theorem for all characters}, we note that the case $\beta \rightarrow \infty$ of \eqref{Asymptotic for ratios of all characters} leads to a result concerning the smoothed first moment of values of quadratic Hecke $L$-functions except for a large error term. We shall in fact use the method in this paper to achieve the following valid asymptotic formula in Section \ref{sec4} unconditionally.
\begin{theorem}
\label{Thmfirstmoment}
		Let $w(t)$ be the same as in Theorem~\ref{Theorem for all characters}.  For $1/2>\Re(\alpha)>0$, we have, for any $\varepsilon>0$,
\begin{align}
\label{Asymptotic for first moment}
\begin{split}			
\sum_{n\odd} & L(\tfrac 12+\alpha,\chi_{(1+i)^2n})w\left( \frac {N(c)}X \right) \\
=& X\M w(1)\frac{\pi \zeta_K^{(2)}(1+2\alpha)}{8\zeta_K^{(2)}(2+2\alpha)}+X^{1-\alpha}\M w(1-\alpha)\frac{2^{2\alpha-1} \pi^{2\alpha+1}\Gamma(1-2\alpha)\Gamma (\alpha)}{3\Gamma (1-\alpha)\Gamma (2\alpha)}\cdot\frac{\zeta_K(1-2\alpha)}{\zeta_K(2)}+O\lz X^{1/2+\varepsilon}\pz.
\end{split}
\end{align}
\end{theorem}

  We note here that the main term in \eqref{Asymptotic for first moment} are identified with exactly the ones obtained by taking $\beta \to \infty$ in the main term of \eqref{Asymptotic for ratios of all characters}. As the error term in \eqref{Asymptotic for first moment} is uniform for $\alpha$, we can further take the limit $\alpha \rightarrow 0^+$ and deduce the following asymptotic formula for the smoothed first moment of central values of quadratic Hecke $L$-functions.
\begin{corollary}
\label{Thmfirstmomentatcentral}
		Let $w(t)$ the same as in Theorem~\ref{Theorem for all characters}. We have, for any $\varepsilon>0$,
\begin{align*}
%%\label{Asymptotic for first moment at central}
\begin{split}			
	& 	\sum_{n\odd} L(\tfrac 12,\chi_{(1+i)^2n})w\left( \frac {N(c)}X \right) = XQ(\log X)+O\lz  X^{1/2+\varepsilon}\pz.
\end{split}
\end{align*}
  where $Q$ is a linear polynomial whose coefficients depend only on the absolute constants and $\M w(1)$ and $\M w'(1)$.
\end{corollary}

We omit the explicit expression of $Q$ here as our main focus here is the error term.  One can compute it by working with the function $A(s,w)$ defined in \eqref{Aswexp} which has a double pole at $s=1, w=1/2$.  Evaluation of the first moment of central values of the classical quadratic family of Dirichlet $L$-functions goes back to M. Jutila \cite{Jutila}. The error terms in Jutila's results were subsequently improved in \cites{ViTa, DoHo, Young1}. A result of D. Goldfeld and J. Hoffstein in \cite{DoHo} implies that the error term is of size $O(X^{1/2 + \varepsilon})$ for a smoothed first moment, a result which was later obtained by M. P. Young in \cite{Young1} using a recursive approach. We point out here that a result similar to Corollary \ref{Thmfirstmomentatcentral} has been achieved by the first-named author \cite{Gao20} using Young's method. \newline

  In \cite{DGH}, A. Diaconu, D. Goldfeld and J. Hoffstein applied the method of multiple Dirichlet series to obtain  an asymptotical formula for the third moment of quadratic twists of Dirichlet $L$-functions. It is also pointed out by them that there were many advantages to treat multiple Dirichlet series as functions of several complex variables. The general philosophy is then to develop enough functional equations for a corresponding multiple Dirichlet series to obtain meromorphic continuations of the series to a region as large as possible. A typical way of obtaining the desired functional equations is to make use of the available functional equations for $L$-functions associated to primitive characters and then to modify the multiple Dirichlet series involved appropriately so that such functional equations hold for $L$-functions associated to non-primitive characters as well. See \cite{D19} and \cite{DW21} for further applications of this approach in both function fields and number fields setting. \newline

  In comparison, \v Cech obtained functional equations for the multiple Dirichlet series studied in \cite{Cech1} for all characters by modifying the usual approach to acquire such functional equations for $L$-functions attached to primitive characters. This process involves with a type of Poisson summation formula developed by K. Soundararajan in \cite{sound1} when studying the non-vanishing of the central values of quadratic Dirichlet $L$-functions. The advantage of \v Cech's method is that one no longer needs to modify the multiple Dirichlet series involved to obtain the desired functional equations. Our result in this paper can be viewed as another application of the approach of \v Cech by adopting both multiple Dirichlet series and harmonic analysis in the setting of quadratic twists of Hecke $L$-functions in the Gaussian field. \newline

Finally, we remark here that a result similar to Corollary \ref{Thmfirstmomentatcentral} was obtained in \cite[Theorem 1.1]{G&Zhao8}, also utilizing double Dirichlet series.  Our approach here in establishing Theorem \ref{Thmfirstmoment} (and thus Corollary \ref{Thmfirstmomentatcentral}) is however different and hence represents another proof of such a result.

\section{Preliminaries}
\label{sec 2}

We first include some auxiliary results needed in the paper.

%%----------------------------------------------------------------------------
\subsection{Quadratic Hecke characters and quadratic Gauss sums}
\label{sec2.4}
%%----------------------------------------------------------------------------
   Recall that $K=\mq(i)$ and it is well-known that $K$ has class number one. Set $U_K=\{ \pm 1, \pm i \}$ and $D_{K}=-4$ for the group of units in $\mathcal{O}_K$ and the discriminant of $K$, respectively. We reserve the letter $\varpi$ for a prime in $\mathcal{O}_K$ by which we mean that $(\varpi)$ is a prime ideal.  We say an element $n \in \mathcal{O}_K$ is odd if $(n,2)=1$ and an element $n \in \mathcal{O}_K$ is square-free if no prime square divides $n$. Note that $n$ is square-free if and only if $\mu_{[i]}(n) \neq 0$, where $\mu_{[i]}$ is the M\"obius function on $\mathcal{O}_K$. Notice that $(1+i)$ is the only prime ideal in $\mathcal{O}_K$ that lies above the integral ideal $(2)$ in $\mz$. \newline

   For $q \in \mathcal{O}_K$, let $\left (\mathcal{O}_K / (q) \right )^*$ denote the group of reduced residue classes modulo $q$ which is the multiplicative group of invertible elements of $\mathcal{O}_K / (q)$. Note that $\left (\mathcal{O}_K / ((1+i)^3) \right )^*$ is isomorphic to the cyclic group of order four generated by $i$. This implies that every ideal in $\mathcal{O}_K$ co-prime to $2$ has a unique generator congruent to $1$ modulo $(1+i)^3$.  This generator is called primary. Furthermore, $\left (\mathcal{O}_K / (4) \right )^*$ is isomorphic to the direct product of a cyclic group of order $4$ and a cyclic group of order $2$. More precisely, we have
\begin{align*}
%%\label{resclassmod4}
  \left (\mathcal{O}_K / (4) \right )^* \cong  <i> \times <1+ 2(1+i)>.
\end{align*}
  Note that the class $1 \pmod {(1+i)^3}$ gives rise to two classes $1+2(1+i)k, k \in \{0,1 \}$ modulo $4$. It follows from this that an element $n=a+bi \in \mathcal{O}_K$ with $a, b \in \mz$ is primary if and only if $a \equiv 1 \pmod{4}, b \equiv
0 \pmod{4}$ or $a \equiv 3 \pmod{4}, b \equiv 2 \pmod{4}$. This result can be found in Lemma 6 on \cite[p. 121]{I&R}.  We shall refer to these two cases above as $n$ is of type $1$ and type $2$, respectively. \newline

   Let $0 \neq q \in \mathcal{O}_K$ and
\begin{align*}
%%\label{chi}
  \chi: \left (\mathcal{O}_K / (q) \right )^*  \rightarrow S^1 :=\{ z \in \mc :  |z|=1 \}
\end{align*}
be a homomorphism.  Following the nomenclature of \cite[Section 3.8]{iwakow}, we refer $\chi$ as a Dirichlet character modulo $q$. We say that such a Dirichlet character $\chi$ modulo $q$ is primitive if it does not factor through $\left (\mathcal{O}_K / (q') \right )^*$ for any divisor $q'$ of $q$ with $N(q')<N(q)$. \newline

   For any odd $n \in \mathcal{O}_{K}$, the symbol $\leg {\cdot}{n}$ stands for the quadratic residue symbol modulo $n$ in $K$. For an odd prime $\varpi \in \mathcal{O}_{K}$, the quadratic symbol is defined for $a \in \mathcal{O}_{K}$, $(a, \varpi)=1$ by $\leg{a}{\varpi} \equiv a^{(N(\varpi)-1)/2} \pmod{\varpi}$, with $\leg{a}{\varpi} \in \{\pm 1 \}$.  If $\varpi | a$, we define $\leg{a}{\varpi} =0$.  Then the quadratic symbol is extended to any odd composite $n$  multiplicatively. We further define $\leg {\cdot}{c}=1$ for $c \in U_K$. \newline

  The following quadratic reciprocity law (see \cite[(2.1)]{G&Zhao4}) holds for two co-prime primary elements $m, n \in \mathcal{O}_{K}$:
\begin{align}
\label{quadrec}
 \leg{m}{n} = \leg{n}{m}.
\end{align}

    Moreover, we deduce from Lemma 8.2.1 and Theorem 8.2.4 in \cite{BEW} that the following supplementary laws hold for primary $n=a+bi$ with $a, b \in \mz$:
\begin{align}
\label{supprule}
  \leg {i}{n}=(-1)^{(1-a)/2} \qquad \mbox{and} \qquad  \hspace{0.1in} \leg {1+i}{n}=(-1)^{(a-b-1-b^2)/4}.
\end{align}

  Note that the quadratic symbol $\leg {\cdot}{n}$ defined above is a Dirichlet character modulo $n$. It follows from \eqref{supprule} that the quadratic symbol $\psi_i:=\leg {i}{\cdot}$ defines a primitive Dirichlet character modulo $4$.  Also,  \eqref{supprule} implies that  the quadratic symbol $\psi_{1+i}:=\leg {1+i}{\cdot}$ defines a primitive Dirichlet character modulo $(1+i)^5$. This can be inferred by noting that
\begin{align*}
%%\label{resclassmod1+i5}
  \left (\mathcal{O}_K / ((1+i)^5) \right )^* \cong  \langle i \rangle \times \langle 1+ 2(1+i) \rangle \times \langle 5 \rangle.
\end{align*}
As $5 \equiv 1 \pmod 4$ and $\psi_{1+i}(5)=-1$, we get that $\psi_{1+i}$ must be primitive. \newline

Furthermore, we observe that there exists a primitive quadratic Dirichlet character $\psi_2$ modulo $2$ since $\left (\mathcal{O}_K / (2) \right )^*$ is isomorphic to the cyclic group of order two generated by $i$.  Then we have $\psi_2(n)=-1$ for $n \equiv i \pmod 2$. \newline

   For any $l \in \mz$ with $4 | l$, we define a unitary character $\chi_{\infty}$ from $\mc^*$ to $S^1$ by:
\begin{align*}
%%\label{chiinf}
  \chi_{\infty}(z)=\leg {z}{|z|}^l.
\end{align*}
  The integer $l$ is called the frequency of $\chi_{\infty}$. \newline

  Now, for a given Dirichlet character $\chi$ modulo $q$ and a unitary character $\chi_{\infty}$, we can define a Hecke character $\psi$ modulo $q$ (see \cite[Section 3.8]{iwakow}) on the group of fractional ideals $I_K$ in $K$, such that for any $(\alpha) \in I_K$,
\begin{align*}
%%\label{psi}
  \psi((\alpha))=\chi(\alpha)\chi_{\infty}(\alpha).
\end{align*}
If $l=0$, we say that $\psi$ is a Hecke character modulo $q$ of trivial infinite type. In this case, we may regard $\psi$ as defined on $\mathcal{O}_K$ instead of on $I_K$, setting $\psi(\alpha)=\psi((\alpha))$ for any $\alpha \in \mathcal{O}_K$. We may also write $\chi$ for $\psi$ as well, since we have $\psi(\alpha)=\chi(\alpha)$ for any $\alpha \in \mathcal{O}_K$. We then say such a Hecke character $\chi$ is primitive modulo $q$ if $\chi$ is a primitive Dirichlet character. Likewise, we say that  $\chi$ is induced by a primitive Hecke character $\chi'$ modulo $q'$ if $\chi(\alpha)=\chi'(\alpha)$ for all
$(\alpha, q')=1$. \newline

   We now define an abelian group $\text{CG}$ such that it is generated by three primitive quadratic Hecke characters of trivial infinite type with corresponding moduli dividing $(1+i)^5$.  More precisely,
\begin{align*}
%%\label{CG}
  \text{CG}=\{ \psi_j : j=1, i, 1+i, i(1+i) \},
\end{align*}
and the commutative binary operation on $\text{CG}$ is given by $\psi_i \cdot \psi_{i(1+i)}=\psi_{1+i}$, $\psi_{1+i} \cdot \psi_{i(1+i)}=\psi_i$ and $\psi_j \cdot \psi_j=\psi_1$ for any $j$. As we shall only evaluate the related characters at primary elements in $\mathcal{O}_K$, the definition of such a product is therefore justified. \newline

  We note that the product of $\chi_c$ for any primary $c$ with any $\psi_j \in \text{CG}$ gives rise to a Hecke character of trivial infinite type. To determine the primitive Hecke character that induces such a product, we observe that every primary $c$ can be written uniquely as
\begin{align*}
%%\label{cdcomp}
 c=c_1c_2, \quad \text{$c_1, c_2$ primary and $c_1$ square-free}.
\end{align*}
The above decomposition allows us to conclude that if $c_1$ is of type 1, then $\chi_c \cdot \psi_j$ for $j \in \{1, i, 1+i, i(1+i)\}$ is induced by the primitive Hecke character $\leg {\cdot}{c_1}\cdot \psi_j$ with modulus $c_1, 4c_1, (1+i)^5c_1$ and $(1+i)^5c_1$ for $j=1, i, 1+i$ and $i(1+i)$, respectively. This is because that $\leg {\cdot }{c_1}$ is trivial on $U_K$ by \eqref{supprule}. Similarly, if $c_1$ is of type 2,  $\chi_c \cdot \psi_j$ for $j \in \{1, i, 1+i, i(1+i) \}$ is induced by the primitive Hecke character $\psi_j \cdot \psi_2 \cdot \leg {\cdot}{c_1}$ with modulus $2c_1, 4c_1, (1+i)^5c_1$ and $(1+i)^5c_1$ for $j=1, i, 1+i$ and $i(1+i)$, respectively. We make the convention that for any primary $n \in \mathcal O_K$, we shall use $\chi_n\cdot \psi_j$ for any $\psi_j \in \text{CG}$ to denote the corresponding primitive Hecke character $\chi$ that induces it throughout the paper. \newline

  For any complex number $z$, we define
\begin{align*}
 \widetilde{e}(z) =\exp \left( 2\pi i  \left( \frac {z}{2i} - \frac {\bar{z}}{2i} \right) \right) .
\end{align*}
With any $r\in \mathcal{O}_K$, the quadratic Gauss sum $g(r, \chi)$ associated to any quadratic Dirichlet character $\chi$ modulo $q$ is defined by
\begin{align*}
%%\label{g2}
 g(r,\chi) = \sum_{x \shortmod{q}} \chi(x) \widetilde{e}\leg{rx}{q}.
\end{align*}
  For the special case when $\chi=\leg {\cdot}{n}$ of any primary $n$, we further define
\begin{align*}
%%\label{g2n}
  g(r,n) = \sum_{x \shortmod{n}} \leg{x}{n} \widetilde{e}\leg{rx}{n}.
\end{align*}

   Let $\varphi_{[i]}(n)$ be the number of elements in $(\mathcal{O}_K/(n))^*$.  We recall from \cite[Lemma 2.2]{G&Zhao4} the following explicitly evaluations of $g(r,n)$ for primary $n$.
\begin{lemma}
\label{Gausssum}
\begin{enumerate}[(i)]
\item  We have
\begin{align*}
%%\label{2.7}
 g(rs,n) & = \overline{\leg{s}{n}} g(r,n), \qquad (s,n)=1, \\
   g(k,mn) & = g(k,m)g(k,n),   \qquad  m,n \text{ primary and } (m , n)=1 .
\end{align*}
\item Let $\varpi$ be a primary prime in $\mathcal{O}_K$. Suppose $\varpi^{h}$ is the largest power of $\varpi$ dividing $k$. (If $k = 0$ then set $h = \infty$.) Then for $l \geq 1$,
\begin{align*}
g(k, \varpi^l)& =\begin{cases}
    0 \qquad & \text{if} \qquad l \leq h \qquad \text{is odd},\\
    \varphi_{[i]}(\varpi^l) \qquad & \text{if} \qquad l \leq h \qquad \text{is even},\\
    -N(\varpi)^{l-1} & \text{if} \qquad l= h+1 \qquad \text{is even},\\
    \leg {ik\varpi^{-h}}{\varpi}N(\varpi)^{l-1/2} \qquad & \text{if} \qquad l= h+1 \qquad \text{is odd},\\
    0, \qquad & \text{if} \qquad l \geq h+2.
\end{cases}
\end{align*}
\end{enumerate}
\end{lemma}

   For any primary $c \in \mathcal O_K$, the Chinese remainder theorem implies that $x = 2y +c z$ varies over the residue class modulo $2c$ as $y$ and $z$ vary over the residue class modulo $c$ and $2$, respectively. From this, we infer that for $j=1$ and $2$,
\begin{align*}
%%\label{g2ctype2}
 g \Big( r,\psi_j \cdot \leg {\cdot}{c} \Big) = \psi_j(c)\leg {2}{c}g(r, \psi_{j})g(r, c).
\end{align*}
   Observe that we have $\leg {2}{c}=\leg {-i(1+i)^2}{c}=\leg {i}{c}$ and that $\psi_j(c)=1$ since we have $c \equiv 1 \pmod 2$. Also, using that the representation of $\left (\mathcal{O}_K / (2) \right )^*$ can be chosen to consist of $1, i$, we see via a direct computation that $g(r, \psi_{j})=(-1)^{\Im(r)}+(-1)^{\Re (r)+j-1}$. It follows that
\begin{align}
\label{g2exp}
 g\left( r,\psi_j \cdot \leg {\cdot}{c} \right) = \leg {i}{c} \Big((-1)^{\Im (r)}+(-1)^{\Re (r)+j-1}\Big )g(r, c).
\end{align}

%%----------------------------------------------------------------
\subsection{Functional equations of $L(s,\chi)$}
\label{sect: apprfcneqn}
%%-----------------------------------------------------------------

  Let $\chi$ be a primitive Hecke character of trivial infinite type modulo $q$. A well-known result of E. Hecke shows that $L(s, \chi)$ has an
analytic continuation to the whole complex plane and satisfies the
functional equation (see \cite[Theorem 3.8]{iwakow})
\begin{align}
\label{fneqn}
  \Lambda(s, \chi) = W(\chi)\Lambda(1-s, \chi), \; \mbox{where} \;  W(\chi) = g(1, \chi)(N(q))^{-1/2}
\end{align}
and
\begin{align}
\label{Lambda}
  \Lambda(s, \chi) = (|D_K|N(q))^{s/2}(2\pi)^{-s}\Gamma(s)L(s, \chi).
\end{align}

 Note that $|W(\chi)|=1$ and that $D_K=-4$ in our situation.  We then conclude from \eqref{fneqn} and\eqref{Lambda} that
\begin{align}
\label{fneqnL}
  L(s, \chi)=W(\chi)N(q)^{1/2-s}\pi^{2s-1}\frac {\Gamma(1-s)}{\Gamma (s)}L(1-s, \chi).
\end{align}

Now Stirling's formula yields
\begin{align}
\label{Stirlingratio}
  \frac {\Gamma(1-s)}{\Gamma (s)} \ll (1+|s|)^{1-2\Re (s)}.
\end{align}

  For our purpose here, we need to generalize the above functional equation to non-primitive quadratic characters.  To that end, we first recall that the Mellin transform $\hat{f}$ of any function $f$ is defined to be
\begin{align*}
     \widehat{f}(s) =\int\limits^{\infty}_0f(t)t^s\frac {\dif t}{t}.
\end{align*}

The following Poisson summation formula for smoothed character sums over all elements in $\mathcal{O}_K$ is \cite[Lemma 2.7]{G&Zhao4}.
\begin{lemma}
\label{Poissonsum} Let $\chi$ be a Dirichlet character modulo $n$. For any smooth function $W:\mr^{+} \rightarrow \mr$,  we have for $X>0$,
\begin{align}
\label{PoissonsumQw}
   \sum_{m \in \mathcal{O}_K}\chi(m)W\left(\frac {N(m)}{X}\right)=\frac {X}{N(n)}\sum_{k \in
   \mathcal{O}_K}g(k,\chi)\widetilde{W} \left(\sqrt{\frac {N(k)X}{N(n)}}\right),
\end{align}
    where
\begin{align*}
%%\label{Wtdef}
   \widetilde{W}(t) =& \int\limits^{\infty}_{-\infty}\int\limits^{\infty}_{-\infty}W(N(x+yi))\widetilde{e}\left(- t(x+yi)\right)\dif x \dif y, \quad t \geq 0.
\end{align*}
\end{lemma}

Applying the above lemma with $W(s)=e^{-s}$, we first compute
\begin{align*}
%%\label{etdef}
   \widetilde{W}(t) =& \int\limits^{\infty}_{-\infty}\int\limits^{\infty}_{-\infty}e^{-(x^2+y^2)}\widetilde{e}\left(- t(x+yi)\right)\dif x \dif y =\int\limits^{\infty}_{-\infty}\int\limits^{\infty}_{-\infty}e^{-(x^2+y^2)-2\pi i yt} \dif x \dif y =\pi e^{-\pi^2t^2}.
\end{align*}
Now it follows from this and \eqref{PoissonsumQw} that we have
\begin{align*}
%%\label{PoissonsumWet}
   \sum_{m \in \mathcal{O}_K}\chi(m)\exp \left(-2\pi y N(m)\right)=\frac {1}{2 y N(n)}\sum_{k \in
   \mathcal{O}_K}g(k,\chi)\exp \left(-\frac {\pi N(k)}{2 y N(n)}\right).
\end{align*}

  We apply the above identity to the Dirichlet character $\widetilde{\chi_n}$ modulo $2n$ defined by $\widetilde{\chi_n}=\psi_j \cdot \leg {\cdot}{n}$ when $n$ is of type $j$ for $j=1,2$. We note that $\widetilde{\chi_n}(0)=0$ and by Lemma \ref{Gausssum} that $g(0,n)=0$ unless $n=\square$, where we write $\square$ for a perfect square in $\mathcal O_K$. We thus conclude for $n \neq \square$,
\begin{align}
\label{PoissonsumWet1}
   \sum_{\substack{m \neq 0 \\ m \in \mathcal{O}_K}}\widetilde{\chi_n}(m)\exp \left(-2\pi y N(m)\right)=\frac {1}{2 y N(2n)}\sum_{\substack{ k \neq 0 \\ k \in
   \mathcal{O}_K}}g(k,\widetilde{\chi_n})\exp \left(-\frac {\pi N(k)}{2 y N(2n)}\right).
\end{align}

   We consider the Mellin transform of the left-hand side above.  This leads to
\begin{equation} \label{keyleft}
	\int\limits_{0}^{\infty}y^{s}\sum_{\substack{m \neq 0 \\ m \in \mathcal{O}_K}}\widetilde{\chi_n}(m)\exp \left(-2\pi y N(m)\right)\Big )\frac{\dif y}{y}=\sum_{\substack{m \neq 0 \\ m \in \mathcal{O}_K}}\widetilde{\chi_n}(m)\int\limits_{0}^{\infty}y^{s}e^{-2\pi N(m)y}\frac{\dif y}{y}=4(2\pi)^{-s}\Gamma(s) L(s,\widetilde{\chi_n}),
\end{equation}
  where the last equality above follows by noting that $\widetilde{\chi_n}(c)=1$ for $c \in U_K$ and hence can be regarded as a Hecke character modulo $2n$ of trivial infinite type. \newline

   On the other hand, the Mellin transform of the right-hand side of \eqref{PoissonsumWet1} equals
\begin{align}
\begin{split}
\label{keyright}
  & \frac {1}{2 N(2n)} \int\limits_{0}^{\infty}y^{s-1}\sum_{\substack{ k \neq 0 \\ k \in
   \mathcal{O}_K}}g(k,\widetilde{\chi_n})\exp \left(-\frac {\pi N(k)}{2 y N(2n)}\right)\Big )\frac{\dif y}{y} = \frac {1}{2 N(2n)} \sum_{\substack{ k \neq 0 \\ k \in
   \mathcal{O}_K}}g(k,\widetilde{\chi_n}) \int\limits_{0}^{\infty}y^{s-1}\exp \left(-\frac {\pi N(k)}{2 y N(2n)}\right)\frac{\dif y}{y}.
\end{split}
\end{align}
The change of variable $u=N(k)/(4 y N(2n))$ transforms the last integral above to
\begin{align}
\begin{split}
\label{integaleval}
 \Big ( \frac {N(k)}{4 N(2n)} \Big )^{s-1} \int\limits_{0}^{\infty}u^{1-s}\exp \left(-2 \pi u \right)\frac{\dif u}{u} =\Big ( \frac {N(k)}{4 N(2n)} \Big )^{s-1} (2\pi)^{-(1-s)}\Gamma(1-s).
\end{split}
\end{align}  	

   We deduce from  \eqref{keyleft}--\eqref{integaleval} that
\begin{align*} \begin{split}
%%\label{fcneqn}
  & 4(2\pi)^{-s}\Gamma(s) L(s,\widetilde{\chi_n})=\frac {(4 N(2n))^{1-s}}{2 N(2n)} (2\pi)^{-(1-s)}\Gamma(1-s) \sum_{\substack{ k \neq 0 \\ k \in
   \mathcal{O}_K}}\frac {g(k,\widetilde{\chi_n})}{N(k)^{1-s}}.
\end{split}
\end{align*}

The following result summarizes our discussions above.
\begin{proposition}
\label{Functional equation with Gauss sums}
   For any primary $n \in \mathcal O_K$, $n \neq \square $, we have
\begin{align}
\begin{split}
\label{fcneqnallchi}
  &  L(s,\widetilde{\chi_n})= N(2n)^{-s} \pi^{-(1-2s)}\frac {\Gamma(1-s)}{4\Gamma(s)} \sum_{\substack{ k \neq 0 \\ k \in
   \mathcal{O}_K}}\frac {g(k,\widetilde{\chi_n})}{N(k)^{1-s}}.
\end{split}
\end{align}
\end{proposition}

\subsection{A mean value estimate for quadratic Hecke $L$-functions}
 In the proof of Theorem \ref{Theorem for all characters}, we need the following lemma, a consequence of \cite[Corollary 1.4]{BGL}, which gives an upper bound for the second moment of quadratic Hecke $L$-functions.
\begin{lemma}
\label{lem:2.3}
Suppose $s$ is a complex number with $\Re(s) \geq \frac{1}{2}$ and that $|s-1|>\varepsilon$ for any $\varepsilon>0$. Then we have
\begin{align}
\label{L4est}
\sumstar_{\substack{(m,2)=1 \\ N(m) \leq X}} |L(s,\chi_{m})|^2
\ll (X|s|)^{1+\varepsilon},
\end{align}
  where $\sum^*$ henceforth denotes the sum over square-free elements in $\mathcal{O}_K$.
\end{lemma}

\subsection{Some results on multivariable complex functions}
	
   We include in this section some results from multivariable complex analysis. First we need the notation of a tube domain.
\begin{defin}
		An open set $T\subset\mc^n$ is a tube if there is an open set $U\subset\mr^n$ such that $T=\{z\in\mc^n:\ \Re(z)\in U\}.$
\end{defin}
	
   For a set $U\subset\mr^n$, we define $T(U)=U+i\mr^n\subset \mc^n$.  We have the following Bochner's Tube Theorem \cite{Boc}.
\begin{theorem}
\label{Bochner}
		Let $U\subset\mr^n$ be a connected open set and $f(z)$ be a function that is holomorphic on $T(U)$. Then $f(z)$ has a holomorphic continuation to the convex hull of $T(U)$.
\end{theorem}

The convex hull of an open set $T\subset\mc^n$ is denoted by $\hat T$.  Then we quote the result from \cite[Proposition C.5]{Cech1} concerning the modulus of holomorphic continuations of functions in multiple variables.
\begin{prop}
\label{Extending inequalities}
		Assume that $T\subset \mc^n$ is a tube domain, $g,h:T\rightarrow \mc$ are holomorphic functions, and let $\tilde g,\tilde h$ be their holomorphic continuations to $\hat T$. If  $|g(z)|\leq |h(z)|$ for all $z\in T$, and $h(z)$ is nonzero in $T$, then also $|\tilde g(z)|\leq |\tilde h(z)|$ for all $z\in \hat T$.
\end{prop}

\section{Proof of Theorem \ref{Theorem for all characters}}

The Mellin inversion renders that
\begin{equation}
\label{Integral for all characters}
		\sum_{\substack{n\odd}}\frac{L(\tfrac{1}{2}+\alpha,\chi_{(1+i)^2n})}{L(\tfrac{1}{2}+\beta,\chi_{(1+i)^2n})}w \bfrac {N(n)}X=\frac1{2\pi i}\int\limits_{(2)}A\lz s,\tfrac12+\alpha,\tfrac12+\beta\pz X^s\widehat w(s) \dif s,
\end{equation}
   where
\begin{align}
\label{Aswzexp}
\begin{split}
A(s,w,z)=& \sum_{\substack{n\odd}}\frac{L^{(2)}(w,\chi_n)}{L^{(2)}(z,\chi_n)N(n)^s}
=\sum_{\substack{k,m,n\odd}}\frac{\mu_{[i]}(k)\chi_n(k)\chi_n(m)}{N(k)^zN(m)^wN(n)^s}=\sum_{\substack{m,k\odd}}\frac{\mu_{[i]}(k)L^{(2)}\lz s,\chi_{mk} \pz}{N(m)^wN(k)^z}.
\end{split}
\end{align}
 Here the last equality above follows from the quadratic reciprocity law given in \eqref{quadrec}. \newline

  To deal with the last integral in \eqref{Integral for all characters}, we need to understand the analytical properties of $A(s,w,z)$. We shall do so in the following sections.
	
\subsection{First region of absolute convergence of $A(s,w,z)$}

   We start with the series representation for $A(s,w,z)$ given by the first equality in \eqref{Aswzexp} to see that
\begin{equation*}
%%\label{key}
		A(s,w,z)=\sum_{\substack{n\odd}}\frac{L^{(2)}(w,\chi_n)}{L^{(2)}(z,\chi_n)N(n)^s}=\sum_{\substack{h\odd}}\frac {\prod_{\varpi | h}\big (1-\chi_n(\varpi)N(\varpi)^{-w}\big) }{N(h)^{2s}}\sumstar_{\substack{n\odd}}\frac{L^{(2)}(w,\chi_n)}{L^{(2)}(z,\chi_{nh^2})N(n)^s}.
\end{equation*}

Now the bound
\begin{align*}
%%\label{key}
1-\chi_n(\varpi)N(\varpi)^{-w} \leq 2N(\varpi)^{\max (0,-\Re(w))}
\end{align*}
leads to
\begin{align}
\label{nongenericprodest}
		\prod_{\varpi | h}\big (1-\chi_n(\varpi)N(\varpi)^{-w}\big) \leq 2^{\omega_{[i]}(h)}N(h)^{\max (0,-\Re(w))} \ll N(h)^{\max (0,-\Re(w))+\varepsilon}.
\end{align}
Here $\omega_{[i]}(h)$ denotes the number of distinct primes in $\mathcal O_K$ dividing $h$ and the last estimate above follows from the well-known bound
\begin{align*}
   \omega_{[i]}(h) \ll \frac {\log N(h)}{\log \log N(h)}, \; \mbox{for} \; N(h) \geq 3.
\end{align*}

   We therefore deduce that
\begin{align}
\label{Abound}
		A(s,w,z) \ll \sum_{\substack{h\odd}}\frac {N(h)^{\max (0,-\Re(w))+\varepsilon} }{N(h)^{2s}}\sumstar_{\substack{n\odd}}\frac{L^{(2)}(w,\chi_n)}{L^{(2)}(z,\chi_{nh^2})N(n)^s}.
\end{align}

   We treat the right-hand side expression above by noting the following bound from \cite[Theorem 5.19]{iwakow} that asserts on GRH,
 \begin{align} \label{Lindelof}
\begin{split}
      \frac1{|L(s,\chi_n)|}\ll & (|s|N(n))^{\varepsilon}, \quad \Re(s) \geq 1/2+\varepsilon.
\end{split}
\end{align}
   Moreover,  we deduce from \eqref{L4est} and Cauchy's inequality that for any $\Re(s) \geq 1/2$ and that $|s-1|>\varepsilon>0$,
\begin{align}
\label{L1est}
\sumstar_{\substack{(m,2)=1 \\ N(m) \leq X}} |L(s,\chi_{m})|
\ll X^{1+\varepsilon} |s|^{1/2+\varepsilon}.
\end{align}
   It follows from the above that except for a simple pole at $w=1$, both sums of the right-hand side expression in \eqref{Abound} are convergent for $\Re(s)>1$,  $\Re(w) \geq 1/2$, $\Re(z)>1/2$. \newline

   On the other hand, we apply the functional equation given in \eqref{fneqnL} for $L(w,\chi_n)$ when $\Re(w)<1/2$ to see that both sums of the right-hand side of \eqref{Abound} converge for $\Re(s+w)>3/2$, $\Re(w) < 1/2$, $\Re(z)>1/2$. \newline

  We thus conclude that the function $A(s,w,z)$ converges absolutely in the region
\begin{equation*}
%%\label{key}
		S_0=\{(s,w,z): \Re(s)>1,\ \Re(s+w)>\tfrac{3}{2},\ \Re(z)>\tfrac{1}{2} \}.
\end{equation*}

We also deduce from the last expression of \eqref{Aswzexp} that $A(s,w,z)$ is given by the series
\begin{align}
\label{Sum A(s,w,z) over n}		
A(s,w,z)=&\sum_{\substack{m,k\odd}}\frac{\mu_{[i]}(k)L^{(2)}\lz s,\chi_{mk}\pz}{N(m)^wN(k)^z}=\sum_{\substack{m,k\odd}}\frac{\mu_{[i]}(k)L\lz s,\widetilde\chi_{mk}\pz}{N(m)^wN(k)^z}.
\end{align}

As noted before, $\widetilde\chi_{n}$ is a Hecke character modulo $2n$ of trivial infinite type.  This information will be needed later in the application of the convexity bound and the evaluation of the relevant Gauss sums using Lemma~\ref{Gausssum}.  We write $mk=(mk)_0(mk)^2_1$ with $(mk)_0$ primary and square-free.  The bound in \eqref{nongenericprodest} renders
\begin{align}
\label{Lsmkbound}	
\begin{split}	
L\lz s,\widetilde\chi_{mk}\pz= & \prod_{\varpi | (mk)_1}\big (1-\chi_{(mk)_0}(\varpi)N(\varpi)^{-s}\big) \cdot L\lz s,\widetilde\chi_{(mk)_0}\pz
\ll N((mk)_1)^{\max (0,-\Re(s))+\varepsilon} |L\lz s,\widetilde\chi_{(mk)_0}\pz| \\
\ll & N(mk)^{\max (0,-\Re(s))+\varepsilon} |L\lz s,\widetilde\chi_{(mk)_0}\pz|.
\end{split}
\end{align}

Now we use the convexity bound for $L(s, \widetilde\chi_{(mk)_0})$ (see \cite[Exercise 3, p. 100]{iwakow}) which asserts that
 \begin{align}
\label{Lconvexbound}
\begin{split}
 L\lz s,\widetilde\chi_{(mk)_0}\pz \ll
\begin{cases}
 \left(N((mk)_0)(1+|s|^2) \right)^{(1-\Re(s))/2+\varepsilon}, & 0 \leq \Re(s) \leq 1, \\
 N((mk)_0)^{\varepsilon}, & \Re(s)>1, \ |s-1|>\varepsilon.
\end{cases}
\end{split}
\end{align}

   Applying the above estimations for the case $\Re(s) \geq 1/2$, together with the functional equation in \eqref{fneqnL} for $L\lz s,\widetilde\chi_{mk}\pz $ with $\Re(s)<1/2$ to \eqref{Sum A(s,w,z) over n}, gives that, except for a simple pole at $s=1$ arising from the summands with $mk=\square$, the function $A(s,w,z)$ converges absolutely in the region
\begin{align*}
%%\label{key}
		S_1=& \{(s,w,z):\hbox{$\Re(w)>1,\ \Re(z)>1,\ \Re(s)>1$}\} \bigcup \{(s,w,z):\hbox{$0<\Re(s)<1, \ \Re(s/2+w)>\frac32,\ \Re(s/2+z)>\frac32$}\} \\
   & \hspace*{1cm} \bigcup \{(s,w,z):\hbox{$\Re(s)<0,\ \Re(z)>1,\ \Re(s+w)>\frac32,\ \Re(s+z)>\frac32$}\},
\end{align*}

	The convex hull of $S_0$ and $S_1$ is
\begin{equation}
\label{Region of convergence of A(s,w,z)}
		S_2=\{(s,w,z):\Re(z)> \tfrac{1}{2},\ \Re(s+w)>\tfrac{3}{2},\ \Re(s+z)> \tfrac{3}{2} \}.
\end{equation}
Thus by our discussions above and Theorem \ref{Bochner}, $(s-1)(w-1)A(s,w,z)$ converges has a holomorphic continuation to the region $S_2$.

\subsection{Residue of $A(s,w,z)$ at $s=1$}
	
	It follows from \eqref{Sum A(s,w,z) over n} that $A(s,w,z)$ has a pole at $s=1$ arising from the terms with $mk=\square$. In this case, we have
\begin{equation*}
		L^{(2)}\lz s, \chi_{mk}\pz=\zeta_K(s)\prod_{\varpi|2mk}\lz1-\frac1{N(\varpi)^s}\pz.
\end{equation*}
    Recall that the residue of $\zeta_K(s)$ at $s = 1$ equals $\pi/4$.  Then we have	
\begin{equation*}
%%\label{key}
 \res_{s=1}A(s,w,z)=\frac {\pi}{4} \sum_{\substack{mk=\square \\m,k\odd}}\frac{\mu_{[i]}(k)a_1(2mk)}{N(m)^wN(k)^z}=
\frac{\pi}{8} \sum_{\substack{mk=\square \\m,k\odd}}\frac{\mu_{[i]}(k)a_1(mk)}{N(m)^wN(k)^z},
\end{equation*}
   where we denote by $a_t(n)$ for any $t \in \mc$ the multiplicative function such that $a_t(\varpi^k)=1-1/(N(\varpi)^t)$ for any prime $\varpi$.  We express the last sum as an Euler product and obtain, via a direct computation similar to that done in \cite[(6.8)]{Cech1}, that, for $w=1/2+\alpha$ and $z=1/2+\beta$,
\begin{align}
\label{Residue at s=1}
	\res_{s=1}& A(s, \tfrac{1}{2}+\alpha,\tfrac{1}{2}+\beta) =\frac{\pi \zeta_K^{(2)}(1+2\alpha)}{8\zeta_K^{(2)}(1+\alpha+\beta)}
\prod_{(\varpi,2)=1}\lz1+\frac{N(\varpi)^{\alpha-\beta}-1}{N(\varpi)^{1+\alpha-\beta}(N(\varpi)^{1+\alpha+\beta}-1)}\pz.
\end{align}

\subsection{Second region of absolute convergence of $A(s,w,z)$}

Applying \eqref{Sum A(s,w,z) over n}, together with the functional equation from Proposition \ref{Functional equation with Gauss sums}, gives that
\begin{align}
\begin{split}
\label{A1A2}
 A(s,w,z) =& \sum_{\substack{m,k\odd \\ mk =  \square}}\frac{\mu_{[i]}(k)L\lz s,\widetilde \chi_{mk}\pz}{N(m)^wN(k)^z} +\sum_{\substack{m,k\odd \\ mk \neq \square}}\frac{\mu_{[i]}(k)L\lz s,\widetilde \chi_{mk}\pz}{N(m)^wN(k)^z} \\
=&  \sum_{\substack{m,k\odd \\ mk =  \square}}\frac{\mu_{[i]}(k)\zeta_K(s)\prod_{\varpi | 2mk}(1-N(\varpi)^{-s}) }{N(m)^wN(k)^z} +\sum_{\substack{m,k\odd \\ mk \neq \square}}\frac{\mu_{[i]}(k)L\lz s,\widetilde \chi_{mk}\pz}{N(m)^wN(k)^z} =: A_1(s,w,z)+A_2(s,w,z).
\end{split}
\end{align}

   We observe first that, upon setting $mk=l$,
\begin{align}
\begin{split}
\label{A1swz}
 A_1(s,w,z) =&  \zeta^{(2)}_K(s) \sum_{\substack{l \odd \\ l =  \square}}\frac {\prod_{\varpi | l}(1-N(\varpi)^{-s})}{N(l)^{w}} \sum_{\substack{ k|l \\ k \odd}}\frac{\mu_{[i]}(k)}{N(k)^{z-w}} \\
=& \zeta^{(2)}_K(s) \prod_{(\varpi, 2)=1}\Big (1+ \frac 1{N(\varpi)^{2w}}(1-N(\varpi)^{-s})(1-N(\varpi)^{-(z-w)})(1-N(\varpi)^{-2w})^{-1}\Big ).
\end{split}
\end{align}
  It follows from the above that except for a simple pole at $s=1$, $A_1(s,w,z)$ is holomorphic in the region
\begin{align}
\label{S3}
		S_3=\Bigg\{(s,w,z):\ &\Re(s+2w)>1,\ \Re(w+z)>1, \ \Re(2w)>1, \ \Re(s+w+z) >1 \Bigg\}.
\end{align}

Next, we apply the functional equation \eqref{fcneqnallchi} for $L\lz s,\widetilde \chi_{mk}\pz$ in the case $mk \neq \square$.  This gives
\begin{align}
\begin{split}
\label{Functional equation in s}
 A_2(s,w,z) =\frac{4^{-s}\pi^{2s-1}\Gamma (1-s)}{4\Gamma(s) } C(1-s,s+w,s+z),
\end{split}
\end{align}
  where $C(s,w,z)$ is given by the triple Dirichlet series
\begin{equation*}
%%\label{key}
		C(s,w,z)=\sum_{\substack{q \neq 0 \\m,k\odd \\ mk \neq \square}}\frac{\mu_{[i]}(k)g(q, \widetilde \chi_{mk})}{N(q)^sN(m)^wN(k)^z}=\sum_{\substack{q \neq 0 \\m,k\odd }}\frac{\mu_{[i]}(k)g(q, \widetilde \chi_{mk})}{N(q)^sN(m)^wN(k)^z}-\sum_{\substack{q \neq 0 \\m,k\odd \\ mk=\square }}\frac{\mu_{[i]}(k)g(q, \widetilde \chi_{mk})}{N(q)^sN(m)^wN(k)^z}.
\end{equation*}	

 By \eqref{Region of convergence of A(s,w,z)}, \eqref{S3} and the functional equation \eqref{Functional equation in s}, we see that $C(s,w,z)$ is initially defined in the region
\begin{equation*}
%%\label{key}
		\{(s,w,z):\ \Re(s+2w)>2, \ \Re(s+w)> \tfrac{3}{2}, \ \Re(s+z)> \tfrac{3}{2},\ \Re(w)> \tfrac{3}{2},\ \Re(z)> \tfrac{3}{2} \}.
\end{equation*}
	To extend this region, we exchange the summations in $C(s,w,z)$ and setting $mk=l$ to obtain that
\begin{align}
\label{Cexp}
\begin{split}
  C(s,w,z)=& \sum_{q \neq 0}\frac{1}{N(q)^s}\sum_{\substack{l\odd }}\frac{g\lz q, \widetilde\chi_{l}\pz}{N(l)^w}\sum_{\substack{ k|l \\ k \odd}}\frac{\mu_{[i]}(k)}{N(k)^{z-w}}-\sum_{q \neq 0}\frac{1}{N(q)^s}\sum_{\substack{l\odd \\ l = \square}}\frac{g\lz q, \widetilde\chi_{l}\pz}{N(l)^w}\sum_{\substack{ k|l \\ k \odd}}\frac{\mu_{[i]}(k)}{N(k)^{z-w}} \\
=:& C_1(s,w,z)+C_2(s,w,z).
\end{split}
\end{align}

   We now apply \eqref{g2exp} to evaluate $g\lz q, \widetilde\chi_{l}\pz$ and apply \eqref{supprule} to detect whether $l$ is of type $1$ or not using the character sums $\frac 12 (\chi_1(l) \pm \chi_{i}(l))$. This allows us to recast $C_1(s,w,z)$ and $C_2(s,w,z)$ as
\begin{align}
\label{Cswz}
\begin{split}
  C_1(s,w,z)=& \sum_{q \neq 0}\frac{(-1)^{\Im (q)}}{N(q)^s}\cdot D(w,z-w,q; \chi_i)+\sum_{q \neq 0}\frac{(-1)^{\Re (q)}}{N(q)^s}\cdot D(w,z-w,q; \chi_1),\\
 C_2(s,w,z)=& \sum_{q \neq 0}\frac{(-1)^{\Im (q)}+(-1)^{\Re (q)}}{N(q)^s}D_1(w, z-w,q),
\end{split}
\end{align}
   where
\begin{equation*}
%%\label{key}
	D(w,t,q; \chi) := \sum_{\substack{l\odd}}\frac{\chi(l)g\lz q, l\pz a_{t}(l)}{N(l)^w}, \quad D_1(w,t,q):= \sum_{\substack{l\odd }}\frac{g\lz q, l^2 \pz a_t(l)}{N(l)^{2w}}.
\end{equation*}
  Here $a_t$ is the function defined earlier. \newline

  We have the following result concerning the analytic properties of $D(w,t,q; \chi)$ and $D_1(w,t,q)$.
\begin{lemma}
\label{Estimate For D(w,t)}
 The functions $D(w,t,q; \chi_1), D(w,t,q; \chi_i)$ and $D_1(w,t,q)$ have meromorphic continuations to the region
\begin{equation*}
%%\label{key}
		\{(w,t):\ \Re(w)>1,\ \Re(w+t)>1\},
\end{equation*}
	 except for a simple pole at $w=3/2$ when $q=i \square$ and $t\neq 0$ for $D(w,t,q; \chi_1)$ and a simple pole at $w=3/2$ when $q=\square$ and $t\neq 0$ for $D(w,t,q; \chi_i)$.
 Moreover, for $\Re(w)>1+\varepsilon$ and $\Re(t+w)>1+\varepsilon$, we have, away from the possible pole,
\begin{equation*}
%%\label{key}
			|D(w,t,q; \chi)|\ll|w(t+w)|^{\varepsilon} N(q)^{\max\{\varepsilon,\varepsilon-\Re(t)\}}.
\end{equation*}		
\end{lemma}
\begin{proof}
   As the situations are similar, we consider only the case $D(w,t,q; \chi_i)$ here.  We recast it as
\begin{align}
\label{Dexp}	
	D(w,t,q; \chi_i)=&\prod_{(\varpi, 2)=1} \lz\sum_{k=0}^\infty\frac{ \chi_i(\varpi^k)g\lz q, \varpi^k \pz a_t(\varpi^k)}{N(\varpi)^{kw}}\pz= P_{g}(w,t,q)P_{n}(w,t,q),
\end{align}
	where $P_{g}$ denotes the product over generic odd primes not dividing $q$ and $P_{n}$ denotes the rest.  Lemma \ref{Gausssum} gives
\begin{align}
\label{Pg}
\begin{split}
			P_{g}(w,t,q)=& \prod_{\varpi \nmid 2q}\Big (1+\frac{\leg{q}{\varpi}(1-N(\varpi)^{-t})}{N(\varpi)^{w-1/2}} \Big )=
L^{(2)}( w-\tfrac{1}{2},\chi_{q}) \prod_{\varpi \nmid 2q}\lz 1-\frac {1}{N(\varpi)^{2w-1}}-\frac{\leg{q}{\varpi}}{N(\varpi)^{w+t-1/2}}+\frac{1}{N(\varpi)^{2w+t-1}}\pz
\\
= & \frac {L^{(2)}( w-\tfrac{1}{2},\chi_{q})}{\zeta_K^{(2q)}(2w-1)}\prod_{\varpi \nmid 2q}\lz 1-\frac{\leg{q}{\varpi}}{N(\varpi)^{w+t-1/2}}+O \Big (\frac{1}{N(\varpi)^{2w+t-1}}+\frac{1}{N(\varpi)^{w+t-1/2+2w-1}}\Big ) \pz \\
=& \frac {L^{(2)}(w-\tfrac{1}{2},\chi_{q})}{\zeta_K^{(2q)}(2w-1)L\lz t+w-\tfrac{1}{2},\chi_{q}\pz}\prod_{\varpi \nmid 2q}\lz 1+O \Big (\frac{1}{N(\varpi)^{2w+t-1}}+\frac{1}{N(\varpi)^{3w+t-3/2}} + \frac{1}{N(\varpi)^{2w+2t-1}} \Big ) \pz.
\end{split}
\end{align}
	 The last Euler product is absolutely convergent and $\ll 1$ for $\Re(t+w)>1+\varepsilon,\ \Re(w)>1+\varepsilon,$. This completes the proof of the first assertion of the lemma.  The second assertion of the lemma here can be proved in a manner similar to \cite[Lemma 6.1 (2)]{Cech1}.
\end{proof}
	
  The above lemma now allows us to extend $C(s,w,z)$ to the region
\begin{equation*}
%%\label{key}
		\{(s,w,z):\ \Re(w)>1,\ \Re(z)>1,\ \Re(s)+\min\{0,\Re(z-w)\}>1\}.
\end{equation*}
	Using \eqref{A1A2}--\eqref{Functional equation in s} and the above, we can extend $(s-1)(w-1)(s+w-3/2)A(s,w,z)$ to the region
\begin{equation*}
%%\label{key}
		S_4=\{(s,w,z):\ \Re(s+2w)>1,\ \Re(w+z)>1, \ \Re(2w)>1, \ \Re(s+w)>1,\ \Re(s+z)>1,\ \Re(1-s)+\min\{0,\Re(z-w)\}>1\}.
\end{equation*}
   Note that the condition $\Re(s+2w)>1$ is superseded by $\Re(2w)>1$ and $\Re(s+w)>1$ so that we have
\begin{equation*}
%%\label{key}
		S_4=\{(s,w,z):\ \Re(w+z)>1, \ \Re(2w)>1, \ \Re(s+w)>1,\ \Re(s+z)>1,\ \Re(1-s)+\min\{0,\Re(z-w)\}>1\}.
\end{equation*}

	Note that the point $(1, 1/2, 1/2)$  is in the closure of $S_2$ and the point $(0, 1, 1)$ is in the closure of $S_4$. We then get that the convex hull of $S_2$ and $S_4$ contains
\begin{align*}
%%\label{Final region of definition for A(s,w,z)}
		S_5=\Bigg\{(s,w,z):\ &\Re(s+2w)>2,\ \Re(s+2z)>2, \Re(s+z)>1, \ \Re(s+w)>1, \ \Re(w)>\tfrac{1}{2}, \ \Re(z)>\tfrac{1}{2} \Bigg\}.
\end{align*}
	
Applying Theorem \ref{Bochner} yields that $(s-1)(w-1)(s+w-3/2)A(s,w,z)$ can be holomorphically continued to the region $S_5$.

\subsection{Residue of $A(s,w,z)$ at $s=3/2-w$}
	
   By Lemma \ref{Estimate For D(w,t)}, we see that $D(w,z-w,q; \chi_i)$ has a pole at $w=3/2$ when $q=\square$ and $z-w\neq0$. To compute the residue, we apply \eqref{Dexp} and \eqref{Pg} to see that when $q=\square$,
\begin{equation*}
%\label{key}
		D(w,t,q; \chi_i)=L^{(2)}( w-1/2,\chi_{q})\prod_{\varpi \nmid 2q}\lz 1-\frac {1}{N(\varpi)^{2w-1}}-\frac{\leg{q}{\varpi}}{N(\varpi)^{w+t-1/2}}+\frac{1}{N(\varpi)^{2w+t-1}}\pz P_{n}(w,t,q).
\end{equation*}
	As the residue of $\zeta_K(s)$ at $s=1$ equals $\pi/4$, the residue at $w=3/2$ of $L\lz w-1/2,\chi_{q} \pz$ for $q=\square$ is
\begin{equation*}
%\label{key}
		\frac {\pi}{4}\prod_{\varpi |2q}\big(1-\frac 1{N(\varpi)}\big).
\end{equation*}
	
	Also, if $q=\square$, we have by Lemma \ref{Gausssum} and a direct computation (similar to that in \cite[(6.46)]{Cech1}) that
\begin{align*}
%%\label{key}
\begin{split}
P_{n}\lz\tfrac32,z-\tfrac32,q \pz= \prod_{\substack{\varpi|q \\ \varpi \odd}}\lz1+\frac {1-N(\varpi)^{3/2-z}}{N(\varpi)}\pz.
\end{split}
\end{align*}
	
   We conclude from the above by another direct computation (with the help of the identity given in \cite[(6.51)]{Cech1}) that
\begin{align*}
%\label{key}
\begin{split}
 & \res_{w=3/2}D(w,z-w,q=\square; \chi_i) = \frac {\pi}{4}\frac {P(z)}{\zeta_K(2)\zeta_K(z-\frac 12)}\frac{2^{z+1/2}}{3\cdot 2^{z-1/2}-2},
\end{split}
\end{align*}	
  where $P(z)$ is given in \eqref{Pz}. \newline

  A similar computation also reveals that
\begin{align}
\label{Resw}
\begin{split}
 & \res_{w=3/2}D(w,z-w,q=i\square; \chi_1) = \frac {\pi}{4}\frac {P(z)}{\zeta_K(2)\zeta_K(z-\frac 12)}\frac{2^{z+1/2}}{3\cdot 2^{z-1/2}-2}.
\end{split}
\end{align}	

  We further observe that for any perfect square $q^2$, $\Im (q^2) \equiv 0 \pmod 2$ and $\Re (iq^2) \equiv 0 \pmod 2$. It follows from this, \eqref{Resw} and \eqref{Cswz} that
\begin{align*}
%\label{key}
\begin{split}
		&\res_{w=3/2}C_1(s,w,z)=\frac{  \pi P(z)}{\zeta_K(2)\zeta_K \lz z-\frac12\pz}\cdot\frac{2^{z+1/2}}{3\cdot 2^{z-1/2}-2}\cdot \zeta_K(2s).
\end{split}
\end{align*}	
	Note that the above is also valid in the case $z=w=3/2$ where there is no pole and the residue is $0$ since $1/\zeta_K(1)=0$. \newline
	
	We now apply \eqref{Functional equation in s} and \eqref{Cexp} to conclude that
\begin{align*}
\begin{split}
			&\res_{s=3/2-w}A(s,w,z)=\frac{\pi^{3-2w}\Gamma ( w-1/2)}{\Gamma (\tfrac{3}{2}-w)}\cdot\frac{P(\tfrac{3}{2}-w+z)\zeta_K(2w-1)}{\zeta_K(2)\zeta_K(1-w+z)}\cdot\frac{2^{z+w-3}}{3\cdot 2^{z+1-w}-2}.
\end{split}
\end{align*}
	Setting $w=1/2+\alpha$ and $z=1/2+\beta$,
\begin{align} \label{Aress}
\begin{split}
			&\res_{s=1-\alpha}A(s,1/2+\alpha,1/2+\beta) =\frac{\pi^{2-2\alpha}\Gamma (\alpha)}{\Gamma (1-\alpha)}\cdot\frac{P(3/2-\alpha+\beta)\zeta_K(2\alpha)}{\zeta_K(2)\zeta_K(1-\alpha+\beta)}\cdot\frac{2^{\alpha+\beta-2}}{3\cdot 2^{1-\alpha+\beta}-2}.
\end{split}
\end{align}
	Note that the functional equation \eqref{fneqnL} implies
\begin{align}
\label{zetafcneqn}
  \zeta_K(2\alpha)=\pi^{4\alpha-1}\frac {\Gamma(1-2\alpha)}{\Gamma (2\alpha)}\zeta_K(1-2\alpha).
\end{align}

   The above allows us to recast the expression in \eqref{Aress} as
\begin{align}
\label{Aress1}
\begin{split}
		&\res_{s=1-\alpha}A(s,1/2+\alpha,1/2+\beta) =\frac{\pi^{2\alpha+1}\Gamma(1-2\alpha)\Gamma (\alpha)}{\Gamma (1-\alpha)\Gamma (2\alpha)}\cdot\frac{P(3/2-\alpha+\beta)\zeta_K(1-2\alpha)}{\zeta_K(2)\zeta_K(1-\alpha+\beta)}\cdot\frac{2^{\alpha+\beta-2}}{3\cdot 2^{1-\alpha+\beta}-2}.
\end{split}
\end{align}

\subsection{Bounding $A(s,w, z)$ in vertical strips}
\label{Section bound in vertical strips}
	
In this section, we give estimations of $|A(s,w,z)|$ in vertical strips, which is needed in order to evaluate the integral in \eqref{Integral for all characters}. \newline
	
	For the previously defined regions $S_j$, we set for any fixed $0<\delta <1/1000$,
\begin{equation*}
%%\label{Definition of S tilde}
		\tilde S_j=S_{j,\delta}\cap\{(s,w,z):\Re(s)>-5/2,\ \Re(w)>1/2-\delta\},
\end{equation*}
	where $S_{j,\delta}= \{ (s,w,z)+\delta (1,1,1) : (s,w,z) \in S_j \} $.  Set
\begin{equation*}
%%\label{key}
		p(s,w)=(s-1)(w-1)(s+w-3/2),
\end{equation*}
  so that $p(s,w)A(s,w,z)$ is an analytic function in the considered regions. We also write $\tilde p(s,w)=1+|p(s,w)|$. \newline

   We consider the expression for $A(s,w, z)$ given in \eqref{Abound} and apply \eqref{Lindelof}, getting that in the region $\tilde S_0$, we have under GRH
\begin{align*}
%%\label{Aswbound}
		|A(s,w, z)| \ll |z|^{\varepsilon}\sumstar_{\substack{n\odd}}\frac{L^{(2)}(w,\chi_n)}{N(n)^{s-\varepsilon}}, \; \mbox{for any} \; \varepsilon >0.
\end{align*}

We note the following bound from \cite[Theorem 5.19]{iwakow} concerning $|L(w,\chi_n)|$, which asserts that under GRH, we have for $\Re(w)\geq1/2$,
\begin{equation}
\label{Lindelof1}
		|(w-1)L(w,\chi_n)|\ll |w-1|(|w|N(n))^{\varepsilon}.
\end{equation}

Now the above bound implies that for $\Re(w) \geq 1/2$ and $\Re(s) \geq 1+2\varepsilon$,
\begin{align*}
%%\label{Lindelof}
\begin{split}
      |p(s,w)A(s,w,z)| \ll \tilde p(s,w)|wz|^{\varepsilon}.
\end{split}
\end{align*}

   On the other hand, we apply the functional equation in \eqref{fneqnL} for $L(w,\chi_n)$ for $\Re(w)<1/2$ and the estimate in \eqref{Stirlingratio} to see that when $\Re(w) < 1/2$ and $\Re(s+w) \geq 3/2+2\varepsilon$ for any $\varepsilon >0$,
\begin{align*}
%%\label{Lindelof}
\begin{split}
      |p(s,w)A(s,w,z)| \ll \tilde p(s,w)|wz|^{\varepsilon}(1+|w|)^{1-2\Re(w)+\varepsilon}.
\end{split}
\end{align*}

   We conclude that in $\tilde S_0$,
\begin{align*}
%%\label{AboundS0}
\begin{split}
      |p(s,w)A(s,w,z)| \ll \tilde p(s,w)|wz|^{\varepsilon}(1+|w|)^{\max \{0, 1-2\Re(w) \}+\varepsilon}.
\end{split}
\end{align*}

   Similarly, using \eqref{Sum A(s,w,z) over n}--\eqref{Lconvexbound}, we see that in the region $\tilde S_1$,
\begin{align*}
%%\label{key}		
	|p(s,w)A(s,w,z)|\ll \tilde p(s,w)(1+|s|)^{\max \{0, 1-\Re(s), 1-2\Re(s)\}+\varepsilon}.
\end{align*}
  Using Proposition \ref{Extending inequalities}, we obtain that in the convex hulls of $\tilde S_2$, $\tilde S_0$ and $\tilde S_1$, we have
\begin{equation}
\label{AboundS2}
		|p(s,w)A(s,w,z)|\ll \tilde p(s,w) |wz|^{\varepsilon}(1+|w|)^{\max \{0, 1-2\Re(w) \}+\varepsilon}(1+|s|)^{6+\varepsilon}.
\end{equation}

   Moreover, by \eqref{A1swz}, the convexity bound given in \eqref{Lconvexbound} for $\zeta_K(s)$, the functional equation given in \eqref{fneqnL} for $\zeta_K(s)$ for $\Re(s)<1/2$ and \eqref{Stirlingratio} that in the region $\tilde S_3$, we have
\begin{align}
\label{A1bound}
		|A_1(s,w,z)| \ll (1+|s|)^{\max \{0, 1-\Re(s), 1-2\Re(s)\}+\varepsilon}.
\end{align}

   Also, by \eqref{Cexp}, \eqref{Cswz} and Lemma \ref{Estimate For D(w,t)}, we have
\begin{equation}
\label{Csbound}
		|(w-3/2)C(s,w,z)|\ll (1+|w-3/2|)|wz|^{\varepsilon}
\end{equation}
   in the region
\begin{equation*}
%%\label{key}
		\{(s,w,z):\Re(w)>1+\varepsilon,\ \Re(z)>1+\varepsilon,\ \Re(s)+\min\{0,\Re(z-w)>1+\varepsilon\},
\end{equation*}

  Now,  applying \eqref{A1A2}, the functional equation \eqref{Functional equation in s}, \eqref{Stirlingratio} and the bounds given in \eqref{A1bound}, \eqref{Csbound}, we obtain that in the region $\tilde S_3$,
\begin{equation}
\label{AboundS3}
		|p(s,w)A(s,w,z)|\ll \tilde p(s,w) |wz|^{\varepsilon}(1+|s|)^{6+\varepsilon}.
\end{equation}

  Note that we have $\Re(w)>1/2$ in the convex hulls of $\tilde S_4$, $\tilde S_2$ and $\tilde S_3$. We then conclude from \eqref{AboundS2},  \eqref{A1bound}, \eqref{AboundS3} and Proposition \ref{Extending inequalities} that in $\tilde S_4$,
\begin{equation}
\label{Abound1}
		|p(s,w)A(s,w,z)|\ll \tilde p(s,w)|wz|^{\varepsilon}(1+|s|)^{6+\varepsilon}.
\end{equation}
	
\subsection{Completion of proof}

We then shift the integral in \eqref{Integral for all characters} to $\Re(s)=N(\alpha,\beta)+\varepsilon$,  where $N(\alpha,\beta)$ is given in \eqref{Nab}. We encounter two simple poles at $s=1$ and  $s=1-\alpha$ in this process, with the corresponding residues being given in \eqref{Residue at s=1} and \eqref{Aress1}. These give the main terms in \eqref{Asymptotic for ratios of all characters}. \newline

  To estimate the integral on the new line, we use \eqref{Abound1} and obtain that on this line
\begin{equation}
\label{Bound in vertical strips}
		|A(s,w,z)|\ll|wz|^{\varepsilon}(1+|s|)^{6+\varepsilon}.
\end{equation}
    Moreover, note that integration by parts implies that for any integer $E \geq 0$,
\begin{align}
\label{whatbound}
 \hat w(s)  \ll  \frac{1}{(1+|s|)^{E}}.
\end{align}
	
   We apply this with \eqref{Bound in vertical strips} to get that the integral on the new line is bounded by the error term given in \eqref{Asymptotic for ratios of all characters}. This completes the proof of Theorem \ref{Theorem for all characters}.
	
\section{Proof of Theorem \ref{Thmfirstmoment}}
\label{sec4}

  The proof of Theorem \ref{Thmfirstmoment} is similar to that of Theorem \ref{Theorem for all characters}.  We shall therefore be brief at places where the arguments are parallel.  The Mellin inversion yields
\begin{equation}
\label{Integral for first moment}
		\sum_{\substack{n\odd}}L( \tfrac{1}{2}+\alpha,\chi_{(1+i)^2n})w \bfrac {N(n)}X=\frac1{2\pi i}\int\limits_{(2)}A\lz s,\tfrac12+\alpha \pz X^s\widehat w(s) \dif s,
\end{equation}
   where
\begin{align}
\label{Aswexp}
\begin{split}
A(s,w)=& \sum_{\substack{n\odd}}\frac{L(w,\chi_n)}{N(n)^s}.
\end{split}
\end{align}

   The above representation implies that $A(s,w)$ can be regarded as the case $z \rightarrow \infty$ of $A(s,w,z)$ defined by the first equality in \eqref{Aswzexp}.
In particular, the function $(s-1)(w-1)A(s,w)$ continues holomorphically in the resulting region by take $z \rightarrow \infty$ in the definition of $S_2$ in \eqref{Region of convergence of A(s,w,z)}, namely,
\begin{equation}
\label{S1}
%%\label{Region of convergence of A(s,w,z)}
		\mathcal S_1=\{(s,w): \Re(s+w)>3/2 \}.
\end{equation}
  Note that we do not need to assume GRH in the above process, as one checks that GRH is only assumed in the estimate in \eqref{Lindelof}, which in turn is applied to bound $1/L^{(2)}(z, \chi_n)$ in \eqref{Abound}, a term not present here. \newline

  On the other hand, we deduce also from \eqref{Aswexp} that
\begin{align}
\label{Aswexp1}
\begin{split}
A(s,w)=& \sum_{\substack{m\odd}}\frac{L\lz s,\widetilde\chi_{m}\pz}{N(m)^w}.
\end{split}
\end{align}

   The above representation implies that $A(s,w)$ can be regarded as the case $k=1$ of $A(s,w,z)$ defined by the last equality in \eqref{Sum A(s,w,z) over n}. We then deduce from the case $k=1$ of \eqref{A1A2} that
\begin{align}
\begin{split}
\label{Asw12}
 A(s,w) =&   \sum_{\substack{m\odd \\ m =  \square}}\frac{\zeta_K(s)\prod_{\varpi | 2m}(1-N(\varpi)^{-s}) }{N(m)^w} +\sum_{\substack{m\odd \\ m \neq \square}}\frac{L\lz s,\widetilde \chi_{m}\pz}{N(m)^w} =: A_1(s,w)+A_2(s,w).
\end{split}
\end{align}

   Here
\begin{align*}
\begin{split}
%%\label{Functional equation in s}
 A_1(s,w) =& \zeta^{(2)}_K(s) \prod_{(\varpi, 2)=1}\Big (1+ \frac 1{N(\varpi)^{2w}}(1-N(\varpi)^{-s})(1-N(\varpi)^{-2w})^{-1}\Big ).
\end{split}
\end{align*}
  Except for a simple pole at $s=1$, $A_1(s,w)$ is holomorphic in the region
\begin{align}
\label{S2-1}
		\mathcal S_2=\Bigg\{(s,w):\ &\Re(s+2w)>1,\ \Re(2w)>1 \Bigg\}.
\end{align}

Now the functional equation \eqref{fcneqnallchi} for $L\lz s,\widetilde \chi_{m}\pz$ in the case $m \neq \square$ leads to
\begin{align}
\begin{split}
\label{A2C}
 A_2(s,w) =\frac{4^{-s}\pi^{2s-1}\Gamma (1-s)}{4 \Gamma(s) } C(1-s,s+w),
\end{split}
\end{align}
  where $C(s,w)$ is given by
\begin{align}
\label{Cdecomp}
\begin{split}
		C(s,w)=& \sum_{\substack{q \neq 0 \\m \odd }}\frac{g(q, \widetilde \chi_{m})}{N(q)^sN(m)^w}-\sum_{\substack{q \neq 0 \\m \odd \\ m=\square }}\frac{g(q, \widetilde \chi_{m})}{N(q)^sN(m)^w} \\
=& \sum_{q \neq 0}\frac{1}{N(q)^s}\sum_{\substack{m \odd }}\frac{g\lz q, \widetilde\chi_{m}\pz}{N(m)^w}-\sum_{q \neq 0}\frac{1}{N(q)^s}\sum_{\substack{m \odd \\ m = \square}}\frac{g\lz q, \widetilde\chi_{m}\pz}{N(m)^w} =: C_1(s,w)+C_2(s,w).
\end{split}
\end{align}	

 By \eqref{S1}, \eqref{S2-1} and the functional equation \eqref{A2C}, $C(s,w)$ is initially defined in the region
\begin{equation*}
%%\label{key}
		\{(s,w):\ \Re(s+2w)>2, \ \Re(w)>\tfrac{3}{2} , \ \Re (s+w) > \tfrac{3}{2} \}.
\end{equation*}
	To extend this region, we apply \eqref{g2exp} to evaluate $g\lz q, \widetilde\chi_{m}\pz$ and apply \eqref{supprule} to detect whether $m$ is of type $1$ or not using the character sums $\frac 12 (\chi_1(m) \pm \chi_{i}(m))$. This allows us to recast $C_1(s,w)$ and $C_2(s,w)$ as
\begin{align}
\label{Csw}
  C_1(s,w)= \sum_{q \neq 0}\frac{(-1)^{\Im (q)}}{N(q)^s} D(w,q; \chi_i)+\sum_{q \neq 0}\frac{(-1)^{\Re (q)}}{N(q)^s} D(w,q; \chi_1) \; \mbox{and} \;
 C_2(s,w)= \sum_{q \neq 0}\frac{(-1)^{\Im (q)}+(-1)^{\Re (q)}}{N(q)^s}D_1(w,q),
\end{align}
   where
\begin{equation*}
%%\label{key}
	D(w, q; \chi) := \sum_{\substack{m \odd}}\frac{\chi(m)g\lz q, m \pz }{N(m)^w} \quad \mbox{and} \quad D_1(w, q):= \sum_{\substack{m \odd }}\frac{g\lz q, m^2 \pz }{N(m)^{2w}}.
\end{equation*}

  We have the following result concerning the analytical property of $D(w, q; \chi)$ and $D_1(w, q)$.
\begin{lemma}
\label{Estimate For D(w)}
 The functions $D(w, q; \chi_1)$, $D(w, q; \chi_i)$ and $D_1(w, q)$ have meromorphic continuations to the region
\begin{equation*}
%%\label{key}
		\{w :\ \Re(w)>1 \},
\end{equation*}
	 except for a simple pole at $w=3/2$ if $q=i\square$ for $D(w, q; \chi_1)$ and a simple pole at $w=3/2$ if $q=\square$ for $D(w, q; \chi_i)$.
 Moreover, for $\Re(w)>1+\varepsilon$, we have, away from the possible pole,
\begin{equation*}
%%\label{key}
			|D(w,q,\chi)|\ll \big ( (1+|w|^2)N(q)\big )^{\max ((3/2-\Re(w))/2, 0)+\varepsilon}.
\end{equation*}		
\end{lemma}
\begin{proof}
   As the situations are similar, we give the details only for the case $D(w, q; \chi_i)$ here.  Recasting it as
\begin{align}
\label{Dexp1}	
	D(w, q; \chi_i)=& \frac {L^{(2)}( w-1/2,\chi_{q})}{\zeta_K(2w-1)}\prod_{\substack{\varpi|2q }}\lz1-\frac 1{N(\varpi)^{2w-1}} \pz^{-1} \prod_{\substack{\varpi^a \| q \\ \varpi \odd}}\lz1+\sum_{k=1}^{\lfloor \frac a 2\rfloor}\frac{\varphi_{[i]}(\varpi^{2k})}{N(\varpi)^{2kw}}+
\frac{\chi_i(\varpi^{a+1})g\lz q, \varpi^{a+1} \pz }{N(\varpi)^{(a+1)w}}\pz ,
\end{align}
the first assertion of the lemma readily follows. \newline

    We further note that when $\Re(w)>1$,
\begin{align*}
%%\label{Dexp}	
\begin{split}
    & \prod_{\varpi|2q}\lz1-\frac 1{N(\varpi)^{2w-1}} \pz^{-1} \ll  N(q)^{\varepsilon}, \\
	& \prod_{\substack{\varpi^a \| q \\ \varpi \odd }}\lz1+\sum_{k=1}^{\lfloor \frac a 2\rfloor}\frac{\varphi_{[i]}(\varpi^{2k})}{N(\varpi)^{2kw}}+
\frac{\chi_i(\varpi^{a+1})g\lz q, \varpi^{a+1} \pz }{N(\varpi)^{(a+1)w}}\pz \ll  \prod_{\varpi |q}\lz1+\sum_{k=1}^\infty N(\varpi)^{2k(1-w)}\pz \ll N(q)^{\varepsilon}.
\end{split}
\end{align*}

   Applying the above together with the estimations given in \eqref{Lconvexbound} to bound $L\lz w-1/2,\chi_{q}\pz$, we see that the second assertion of the lemma follows. This completes the proof.
\end{proof}
	
  The above lemma now allows us to extend $C(s,w)$ to the region
\begin{equation*}
%%\label{key}
		\{(s,w):\ \Re(w)>1,\ \Re(s)+\min \{0, (\Re(w)-3/2)/2 \} >1\}.
\end{equation*}
	Using \eqref{Asw12}--\eqref{Cdecomp} and the above, we can extend $(s-1)(w-1)(s+w-3/2)A(s,w)$ to the region
\begin{equation*}
%%\label{key}
		\mathcal S_3=\{(s,w): \ \Re(2w)>1, \ \Re(s+w)>1, \ \Re(1-s)+\min \{0, (\Re(s+w)-3/2)/2 \}>1\}.
\end{equation*}
	The convex hull of $\mathcal S_1$ and $\mathcal S_3$ contains
\begin{align*}
%%\label{Final region of definition for A(s,w,z)}
	\mathcal S_4=\Bigg\{(s,w): \ \Re(s+w)>1 \Bigg\}.
\end{align*}
	
Now Theorem \ref{Bochner} implies that $(s-1)(w-1)(s+w-3/2)A(s,w)$ can be continued holomorphically to the region $\mathcal S_3$.

\subsection{Residue of $A(s,w)$ at $s=1$}
	
	We see from \eqref{Aswexp1} that $A(s,w)$ has a pole at $s=1$ coming from the summands with $m=\square$.  In this case, we have
\begin{equation*}
%%\label{key}
  \res_{s=1}A(s,w)=\frac {\pi}{8}\sum_{\substack{m=\square \\m\odd}}\frac{a_1(m)}{N(m)^w},
\end{equation*}
   where we recall that $a_1(n)$ denotes the multiplicative function with $a_1(\varpi^k)=1-1/N(\varpi)$. \newline

	Writing the last sum above as an Euler product, we get
\begin{align*}
\begin{split}
 \res_{s=1}A(s,w)=	&\frac {\pi \zeta_K^{(2)}(2w)}{8\zeta_K^{(2)}(2w+1)}.
\end{split}
\end{align*}

	Setting $w=1/2+\alpha$ leads to
\begin{align}
\label{ResidueAsw at s=1}
  \res_{s=1}& A(s,1/2+\alpha) =\frac{\pi \zeta_K^{(2)}(1+2\alpha)}{8\zeta_K^{(2)}(2+2\alpha)}.
\end{align}
	
\subsection{Residue of $A(s,w)$ at $s=3/2-w$}
	
   By Lemma \ref{Estimate For D(w)}, $D(w, q; \chi_i)$ has a pole at $w=3/2$ if $q=\square$. To compute the residue, \eqref{Dexp1} gives that if $q=\square$, then
\begin{equation*}
%\label{key}
		 \frac {D(w,q; \chi_i) \zeta_K(2w-1)}{L^{(2)}( w-1/2,\chi_{q})} \Bigg |_{w=\frac 32}= \prod_{\varpi|2q}\lz1-\frac 1{N(\varpi)^{2}} \pz^{-1}\prod_{\substack{\varpi|q \\ \varpi \odd}}\lz1+\frac {1}{N(\varpi)}\pz .
\end{equation*}

   It follows that
\begin{align*}
%\label{key}
\begin{split}
 & \res_{w=3/2}D(w,q=\square; \chi_i)=\frac {\pi}{6\zeta_K(2)}.
\end{split}
\end{align*}	
 Similarly,
\begin{align*}
%\label{key}
\begin{split}
 & \res_{w=3/2}D(w,q=i\square; \chi_1) = \frac {\pi}{6\zeta_K(2)}.
\end{split}
\end{align*}	

  We deduce from this and \eqref{Csw} together with the observation that $\Im (q^2) \equiv 0 \pmod 2$ and $\Re (iq^2) \equiv 0 \pmod 2$ that
\begin{align*}
%\label{key}
\begin{split}
		&\res_{w=3/2}C_1(s,w)=\frac{ 2 \pi }{3 \zeta_K(2)}\zeta_K(2s).
\end{split}
\end{align*}	
	
	We now apply \eqref{Asw12}, \eqref{A2C} and \eqref{Cdecomp} to conclude that
\begin{align*}
\begin{split}
			&\res_{s=3/2-w}A(s,w)=\frac{2^{2w-1}\pi^{3-2w}\Gamma ( w-1/2)}{6\Gamma (3/2-w)}\cdot\frac{\zeta_K(2w-1)}{\zeta_K(2)}.
\end{split}
\end{align*}
	Substituting $w=1/2+\alpha$ yields that
\begin{align}
\label{Asres}
\begin{split}
			&\res_{s=1-\alpha}A(s,1/2+\alpha) =\frac{2^{2\alpha-1}\pi^{2-2\alpha}\Gamma (\alpha)}{3\Gamma (1-\alpha)}\cdot\frac{\zeta_K(2\alpha)}{\zeta_K(2)}.
\end{split}
\end{align}
	We then apply the functional equation \eqref{zetafcneqn} to recast the expression in \eqref{Asres} as
\begin{align}
\label{Asres1}
\begin{split}
		&\res_{s=1-\alpha}A(s,1/2+\alpha) =\frac{2^{2\alpha-1} \pi^{2\alpha+1}\Gamma(1-2\alpha)\Gamma (\alpha)}{3\Gamma (1-\alpha)\Gamma (2\alpha)}\cdot\frac{\zeta_K(1-2\alpha)}{\zeta_K(2)}.
\end{split}
\end{align}
	
\subsection{Bounding $A(s,w)$ in vertical strips}
\label{Section $A(s,w)$ bound in vertical strips}
	
	We may bound $|A(s,w)|$ in vertical strips following the treatments in Section \ref{Section bound in vertical strips}. Similar to the notations introduced there, we set for the previously defined regions $\mathcal S_j$ and any fixed $0<\delta <1/1000$,
\begin{equation*}
%%\label{Definition of Ssw tilde}
		\widetilde{\mathcal S}_j=\mathcal S_{j,\delta}\cap\{(s,w):\Re(s)>-5/2,\ \Re(w)>1/2-\delta \},
\end{equation*}
	where $\mathcal S_{j,\delta}= \{ (s,w) +\delta (1,1) : (s,w) \in \mathcal S_j \}$. \newline

Replacing the bound in \eqref{Lindelof1} by the one in \eqref{L1est}, gives that unconditionally in $\widetilde{\mathcal S}_2$,
\begin{equation*}
%%\label{AboundS2}
		|p(s,w)A(s,w)|\ll \tilde p(s,w) |wz|^{\varepsilon}(1+|w|)^{\max \{0, 1-2\Re(w) \}+1/2+\varepsilon}(1+|s|)^{6+\varepsilon}.
\end{equation*}

	 By \eqref{Cdecomp}, \eqref{Csw} and Lemma \ref{Estimate For D(w)}, we have
\begin{equation}
\label{Cswbound}
		|(w-3/2)C(s,w)|\ll (1+|w-3/2|) \big ( (1+|w|^2) \big )^{\max ((3/2-\Re(w))/2, 0)+\varepsilon}.
\end{equation}
  in the region
\begin{equation*}
%%\label{key}
		\{(s,w):\Re(w)>1+\varepsilon,\  \Re(s)+\min\{0,(\Re(w)-3/2)/2 \}>1+\varepsilon\}.
\end{equation*}

    Note also that similar to \eqref{A1bound}, in the region $\widetilde {\mathcal S}_2$, we have
\begin{align}
\label{A1swbound}
		|A_1(s,w)| \ll (1+|s|)^{\max \{0, 1-\Re(s), 1-2\Re(s)\}+\varepsilon}.
\end{align}

   It follows from \eqref{Stirlingratio}, \eqref{Asw12}, \eqref{A2C}, \eqref{Cswbound} and \eqref{A1swbound} that in the region $\widetilde {\mathcal S}_3$,
\begin{align*}
%%\label{key}
\begin{split}
		|p(s,w)A(s,w)|\ll &  \tilde p(s,w) (1+|s|)^{6+\varepsilon}\big ( (1+|w+s|^2) \big )^{\max ((3/2-\Re(w+s))/2, 0)+\varepsilon} \\
\ll & \tilde p(s,w) (1+|s|)^{6+\varepsilon}\big ( (1+|w|^2)(1+|s|^2) \big )^{\max ((3/2-\Re(w+s))/2, 0)+\varepsilon} \\
\ll &  \tilde p(s,w) (1+|s|)^{7+\varepsilon}(1+|w|)^{1/2+\varepsilon},
\end{split}
\end{align*}
as $\Re(w+s)>1$ in $\widetilde {\mathcal S}_3$. \newline

 We then conclude by Proposition \ref{Extending inequalities} that in the convex hull $\widetilde {\mathcal S}_4$ of $\widetilde {\mathcal S}_1$ and $\widetilde {\mathcal S}_4$, we have
\begin{equation}
\label{Aswverticalbounds}
		|p(s,w)A(s,w)|\ll \tilde p(s,w)(1+|s|)^{7+\varepsilon}(1+|w|)^{1/2+\varepsilon}.
\end{equation}

\subsection{Conclusion}

We then shift the integral in \eqref{Integral for first moment} to $\Re(s)=1/2+\varepsilon$ to encounter two simple poles at $s=1$ and  $s=1-\alpha$ in the process, with the corresponding residues given in \eqref{ResidueAsw at s=1} and \eqref{Asres1}. These give the main terms in \eqref{Asymptotic for first moment}. We then apply \eqref{whatbound} together with the bound given in \eqref{Aswverticalbounds} to infer that the integral on the new line is bounded by the error term given in \eqref{Asymptotic for first moment}. This completes the proof of Theorem \ref{Thmfirstmoment}.

\vspace*{.5cm}

\noindent{\bf Acknowledgments.}   P. G. is supported in part by NSFC grant 11871082 and L. Z. by the Faculty Silverstar Grant PS65447 at the University of New South Wales.  The authors would also like to thank the anonymous referee for his/her very careful inspection of the paper and many helpful suggestions.

\bibliography{biblio}
\bibliographystyle{amsxport}

\end{document}